\documentclass[11pt,a4paper]{article}

   \usepackage[english]{babel}
   \usepackage[utf8]{inputenc}

\usepackage{hyperref}

\usepackage{amsmath, amsthm, amssymb}
\usepackage{mathtools}
\usepackage{enumerate}
\usepackage{graphicx}           
\usepackage{url}                

\usepackage[all]{xy}
\usepackage{tikz}
\usetikzlibrary{arrows,chains,matrix,positioning,scopes,cd}
\usepackage{thmtools, thm-restate}

\providecommand*{\A}{\mathbb{A}}

\providecommand*{\N}{\mathbb{N}}

\providecommand*{\Z}{\mathbb{Z}}

\providecommand*{\Ab}{\mathbf{Ab}}

\providecommand*{\cA}{\mathcal{A}}
\providecommand*{\cB}{\mathcal{B}}

\providecommand*{\cE}{\mathcal{E}}
\providecommand*{\cF}{\mathcal{F}}

\providecommand*{\cS}{\mathcal{S}}
\providecommand*{\cT}{\mathcal{T}}
\providecommand*{\cU}{\mathcal{U}}
\providecommand*{\cV}{\mathcal{V}}

\providecommand*{\add}{\mathrm{add}}
\providecommand*{\proj}{\mathrm{proj}}

\providecommand*{\ind}{\mathrm{ind}}
\providecommand*{\tol}{\mathrm{torsl}}
\providecommand*{\divbl}{\mathrm{divbl}}
\providecommand*{\fun}{\mathrm{fun}}
\providecommand*{\Fun}{\mathrm{Fun}}
\providecommand*{\fdmod}{\mathrm{mod}}
\providecommand*{\stmod}{\underline{\mathrm{mod}}}

\newcommand*{\stsubm}[2]{\underline{S_{#1}}(#2)}

\providecommand*{\aus}{\mathrm{Aus}}

\providecommand*{\Hom}{\mathrm{Hom}}

\providecommand*{\ext}{\mathrm{Ext}}

\providecommand*{\soc}{\mathrm{soc}}
\providecommand*{\Top}{\mathrm{top}}
\providecommand*{\End}{\mathrm{End}}
\providecommand*{\gen}{\mathrm{gen}}
\providecommand*{\cog}{\mathrm{cogen}}

\providecommand*{\op}{\mathrm{op}}

\providecommand*{\id}{\mathrm{id}}
\providecommand*{\im}{\mathrm{im}}
\providecommand*{\cok}{\mathrm{coker}}

\providecommand*{\stGamma}{\underline{\Gamma}}

\providecommand*{\ob}{\mathrm{Ob}}
\providecommand*{\bil}{\text{-}}

\newtheorem{thm}{Theorem}[section]
\newtheorem{lemma}[thm]{Lemma}

\newtheorem{prp}[thm]{Proposition}
\newtheorem{cor}[thm]{Corollary}
\theoremstyle{remark}
\newtheorem{remark}[thm]{Remark}

\newtheorem*{rem}{Remark}
\theoremstyle{definition}
\newtheorem{defn}[thm]{Definition}

\begin{document}


\author{Ögmundur Eiríksson}
\date{}
\title{From Submodule Categories to the Stable Auslander Algebra}
\maketitle

\begin{abstract}
We construct two functors from the submodule category of a self-injective representation-finite algebra $\Lambda$ to the module category of the stable Auslander algebra of $\Lambda$. 
These functors factor through the module category of the Auslander algebra of $\Lambda$.
Moreover they induce equivalences from the quotient categories of the submodule category modulo their respective kernels and said kernels have finitely many indecomposable objects up to isomorphism.
Their construction uses a recollement of the module category of the Auslander algebra induced by an idempotent and this recollement determines a characteristic tilting and cotilting module.
If $\Lambda$ is taken to be a Nakayama algebra, then said tilting and cotilting module is a characteristic tilting module of a quasi-hereditary structure on the Auslander algebra. We prove that the self-injective Nakayama algebras are the only algebras with this property.
\end{abstract}


\section{Introduction}
\label{sec:intro}
Fix a field $k$, all algebras considered will be algebras over $k$ and all modules will be finite-dimensional left modules unless stated otherwise.
For an algebra $A$ we let $A \bil \fdmod$ denote the category of finite-dimensional left $A$-modules.
As a shorthand notation we often write $(M,N)_{\cA} \coloneqq \Hom_{\cA}(M,N)$ for the homomorphism spaces in an additive category $\cA$. 
Similarly we write $(M,N)_{A} \coloneqq \Hom_{A}(M,N)$ for homomorphisms in $A \bil \fdmod$ for an algebra $A$.

For an additive category $\cA$ let $\ind(\cA)$ denote the class of indecomposable objects in $\cA$ and we write $\ind(A) \coloneqq \ind(A \bil \fdmod)$.
An \emph{additive subcategory} of $\cA$ is a full subcategory closed under finite direct sums and direct summands.
For an object $A \in \cA$ let $\add(A)$ denote the smallest additive subcategory of $\cA$ containing $A$.
Denote the algebra of upper triangular $2 \times 2$ matrices with coefficients in $A$ by $T_2(A)$.

The \emph{category of morphisms} in $A \bil \fdmod$ is the category which has maps\\ $(M_1 \overset{f_M}{\rightarrow} M_0)$ of $A$-modules as objects and a morphism of two objects\\ $(M_1 \overset{f_M}{\rightarrow} M_0)$ and $(N_1 \overset{f_N}{\rightarrow} N_0)$ is a pair $(g_1,g_0) \in (M_1,N_1)_{A} \times (M_0,N_0)_{A}$ such that $f_N g_1 = g_0 f_M$. 
We will identify $T_2(A) \bil \fdmod$ with the category of morphisms in $A \bil \fdmod$.
The \emph{submodule category} of $A$, denoted $\cS(A)$, is the full subcategory of monomorphisms in $T_2(A) \bil \fdmod$.
We denote the full subcategory of epimorphisms by $\cE(A)$.

Studies of submodule categories go back to Birkhoff \cite{Birkhoff1935}.
Recently they have been a subject to active research,
including work of Simson about their tame-wild dichotomy \cite{simson2002,simson2007,simson2015}.
Also Ringel and Schmidmeier have studied their Auslander-Reiten theory \cite{RS2006}, as well as some particular cases of wild type \cite{RS2008}.
Moreover they were studied with respect to Gorenstein-projective modules and tilting theory in \cite{LZ2017,zhang2011}.
The homological properties of submodule categories give extensive information on quiver Grassmannians, in particular their isomorphism classes correspond to strata in certain stratifications \cite{CFR2012}.
If $\Lambda$ is self-injective, the monomorphism category is a Frobenius category, and Chen has shown its stable category is equivalent to the singularity category of $T_2(\Lambda)$ \cite{Chen2011}. 
In \cite{KLM2013} Kussin, Lenzing and Meltzer give a connection of monomorphism categories to weighted projective lines, which again connects them to singularity categories \cite{KLM2013-2}.

We define the \emph{kernel} of an additive functor $F \colon \cA \rightarrow \cB$ as the full subcategory of all objects $A$ in $\cA$ such that $F(A) \cong 0$.
Given an additive category $\cA$ and a full subcategory $\cB$ of $\cA$, the \emph{quotient category} $\cA/\cB$ has objects $\ob(\cA/\cB) = \ob (\cA)$. 
For $X,Y \in \ob(\cA)$, let $R_{\cB}(X,Y)$ denote the morphisms of $\Hom_{\cA}(X,Y)$ that factor through an object in $\cB$. 
The morphisms spaces of $\cA/\cB$ are defined as the quotients $\Hom_{\cA/\cB}(X,Y) \coloneqq \Hom_{\cA}(X,Y)/R_{\cB}(X,Y)$. 

Fix a basic algebra $\Lambda$ of finite representation type. 
Let $E$ be the additive generator of $\Lambda \bil \fdmod$, i.e. the basic $\Lambda$-module such that $\add(E) = \Lambda \bil \fdmod$.
The \emph{Auslander algebra} of $\Lambda$ is $\aus(\Lambda) \coloneqq \mathrm{End}_{\Lambda}(E)^{\op}$.
Write $\Gamma \coloneqq \aus(\Lambda)$ and let $e \in \Gamma$ be the idempotent given by the projection onto the summand $\Lambda$ of $E$.
Write $\Gamma e \Gamma$ for the two sided ideal generated by $e$, the \emph{stable Auslander algebra} is defined as the algebra $\underline{\Gamma} \coloneqq \Gamma/ \Gamma e \Gamma$.

Write $\alpha \coloneqq \cok (E,-)_{\Lambda} \colon T_2({\Lambda}) \bil \fdmod \rightarrow \Gamma \bil \fdmod$,
and let $\epsilon \colon \cS(\Lambda) \to T_2(\Lambda)$ be the quotient functor $\epsilon(M_1 \to M_0) \coloneqq (M_0 \to M_0/M_1)$.
We define the functors $F \coloneqq \stGamma \otimes_{\Gamma} \alpha(-)$ and $G \coloneqq \stGamma \otimes_{\Gamma} \alpha(\epsilon(-))$ from $\cS(\Lambda)$ to $\stGamma \bil \fdmod$.
In \cite{RZ2014} Ringel and Zhang studied those functors in the case $\Lambda = k[x]/\langle x^N \rangle$, and observed that then $\stGamma$ is isomorphic to the preprojective algebra $\Pi_{N-1}$ of type $A_{N-1}$.
They show that each of the functors induces an equivalence of categories $\cS(k[x]/\langle x^N \rangle)/ \chi \rightarrow \Pi_{N-1} \bil \fdmod$, where $\chi$ is the kernel of the respective functor $F$ or $G$ \cite[Theorem 1]{RZ2014}.
Moreover they could also describe those kernels explicitly.
We will prove the following generalization of that statement.

\begin{restatable}{theorem}{induced}
\label{thm:induced}
Let $\Lambda$ be a basic, self-injective and representation finite algebra.
Let $\cU$ denote the smallest additive subcategory of $\cS(\Lambda)$ containing $(E \overset{\id}{\rightarrow} E)$ and all objects of the form $(M \overset{f}{\rightarrow} I)$, where $I$ is a projective-injective $\Lambda$-module.
Also define
$\cV \coloneqq \add(( E \overset{\id}{\rightarrow} E) \oplus (0 \rightarrow E) )$.
Let $m$ be the number of isomorphism classes of $\ind(\Lambda)$.
Then $\cU$ and $\cV$ have $2m$ indecomposable objects up to isomorphism and the following holds.
\begin{enumerate}[$(i)$]
\item
The functor $F$ induces an equivalence of categories $\cS(\Lambda)/\cU \rightarrow \stGamma \bil \fdmod$.
\item
The functor $G$ induces an equivalence of categories $\cS(\Lambda)/\cV \rightarrow \stGamma \bil \fdmod$.
\end{enumerate}
\end{restatable}

For now let us assume $\Lambda$ is self-injective.
The idempotent $e$ yields a recollement of $\Gamma \bil \fdmod$, and there is a tilting and cotilting module $T$ in $\Gamma \bil \fdmod$ given in terms of that recollement.
That recollement and the module $T$ feature in the construction of $F$ and $G$ in Section \ref{sec:from-to}.
Let $\pi \colon \stGamma \bil \fdmod \rightarrow \stGamma \bil \stmod$ be the projection to the stable category and let $\Omega$ denote the syzygy functor on $\stGamma \bil \stmod$.
We prove the following generalization of \cite[Theorem 2]{RZ2014}.

\begin{restatable}{theorem}{difference}
\label{thm:difference}
The functors $\pi  F$ and $\pi G$  differ by the syzygy functor on $\stGamma \bil \fdmod$, more precisely $\pi F = \Omega \pi G$.
\end{restatable}

For the following we don't need $\Lambda$ to be self-injective.
Let $\Gamma \bil \tol$ denote the full subcategory of $\Gamma \bil \fdmod$ given by objects of projective dimension at most 1.
In the final section we prove the following theorem.

\begin{restatable}{theorem}{uniserial}
\label{thm:uniserial}
Let $\Lambda$ be a basic representation-finite algebra and let $\Gamma$ be its Auslander algebra.
Then $\Gamma$ has a quasi-hereditary structure such that the objects of $\Gamma \bil \tol$ are precisely the $\Delta$-filtered $\Gamma$-modules if and only if $\Lambda$ is uniserial.
\end{restatable}

This was already observed in \cite{RZ2014} in the particular case of the Auslander algebra of $k[x]/\langle x^{N} \rangle$.
The theorem implies that $T$ arises as the characteristic tilting module of a quasi-hereditary structure on $\Gamma$ if and only if $\Lambda$ is uniserial, i.e. if $\Lambda$ is a self-injective Nakayama algebra.

The content is organized as follows.
In Section \ref{sec:subm} we recall the notion of a category of finitely presented functors, and that the module category $\Gamma \bil \fdmod$ is equivalent to the category of finitely presented additive contravariant functors from $\Lambda \bil \fdmod$ to the category of abelian groups.
We also give the basic properties of the functor $\alpha$, and
recall characterizations of the projective and injective objects in the category of finitely presented functors.

In Section \ref{sec:self-inj} we restrict our attention to the Auslander algebras of self-injective algebras.
Then we study the recollement induced by the idempotent $e \in \Gamma$ and introduce the tilting and cotilting module $T$.
Moreover we recall some properties of the stable Auslander algebra $\stGamma$.

In Section \ref{sec:from-to} we consider the functors $F$ and $G$ that arise as compositions of functors studied in the previous sections. 
We prove Theorems \ref{thm:induced} and \ref{thm:difference} for these functors and thereby generalise the situation in \cite{RZ2014}.

Section \ref{sec:nakayama} is dedicated to the proof of Theorem \ref{thm:uniserial}. First we give some properties of the Nakayama algebras and recall the notion of a left strongly quasi-hereditary structure.
In Subsection \ref{subsec:QH-Naka} we introduce a quasi-hereditary structure on the Auslander algebras of the Nakayama algebras which fulfils the conditions of Theorem \ref{thm:uniserial}. 
In Subsection \ref{subsec:only-if} we prove that no other Auslander algebras satisfy those conditions.


\section{Submodule Categories}
\label{sec:subm}

Recall that $\Lambda$ is a basic representation-finite algebra.

\begin{defn}
\label{def:submodulecat}
We define functors:
\begin{align*}
	\eta &\colon \cS(\Lambda) \rightarrow T_2({\Lambda}) \bil \fdmod, \quad f \mapsto f, \\
	\epsilon &\colon \cS(\Lambda) \rightarrow T_2({\Lambda}) \bil \fdmod, \quad f \mapsto \cok(f).
\end{align*}
The functor $\eta$ is simply the inclusion of $\cS(\Lambda)$ in $T_2(\Lambda) \bil \fdmod$.
On morphisms, $\epsilon$ is given by the induced maps of cokernels.
Notice that $\epsilon$ is full and faithful and its essential image is the full subcategory $\cE(\Lambda)$ of $T_2(\Lambda)$, hence we can consider $\epsilon$ as a composition of an equivalence $\cS(\Lambda) \rightarrow \cE(\Lambda)$ followed by $\eta$.
\end{defn}

Now we recall some well known facts on representable functors. 
These go back to Auslander \cite{auslander1966}, Freyd \cite{freyd1964,freyd1966} and Gabriel \cite{gabriel1962}, 
while \cite{lenzing1998} contains a handy summary of those techniques.

Let $\cA$ be an additive category.
We consider the category $\Fun(\cA)$ of additive functors from $\cA^{\op}$ to the category $\Ab$ of abelian groups, with morphisms given by natural transformations.
We say a functor $F \in \Fun(\cA)$ is representable if $F$ is isomorphic to $(-,M)_{\cA}$ for some $M \in \ob(\cA)$.
We say $F$ is finitely presented if there exist representable functors $(-,M)_{\cA},(-,N)_{\cA}$ and an exact sequence
\[
\begin{tikzcd}
(-,M)_{\cA} \ar[r] & (-,N)_{\cA} \ar[r] & F \ar[r] & 0.
\end{tikzcd}
\]
We denote the full subcategory of finitely presented functors by $\fun(\cA)$.
The category $\fun(\cA)$ is abelian, cf. \cite[Theorem 5.11]{freyd1964}.
To reduce encumbrance we write $\fun(\Lambda) \coloneqq \fun(\Lambda \bil \fdmod)$.

The following lemma comes from applying \cite[Theorem 5.35]{freyd1964} to the opposite of $\Lambda \bil \fdmod$.

\begin{lemma}
\label{lem:representable}
The functor $M \mapsto (-,M)_{\Lambda}$ from $\Lambda \bil \fdmod$ to $\fun(\Lambda)$ induces an equivalence of categories from $\Lambda \bil \fdmod$ to the full subcategory of representable functors in the category $\fun(\Lambda)$.
Moreover, the representable functors are the projective objects of $\fun(\Lambda)$.
\end{lemma}

Let $D \coloneqq (-,k)_{k}$ be the vector space duality.
If we apply \cite[Theorem 5.35]{freyd1964} to $\Lambda \bil \fdmod$ and then apply the vector space duality we obtain the following dual statement to Lemma \ref{lem:representable}, cf. \cite[Exercise A. Chapter 5]{freyd1964}.

\begin{lemma}
\label{lem:corepresentable}
The functor $M \mapsto D(M,-)_{\Lambda}$ from $\Lambda \bil \fdmod$ to $\fun(\Lambda)$
induces an equivalence from $\Lambda \bil \fdmod$ to the full subcategory of injective objects in $\fun(\Lambda)$.
\end{lemma}


\subsection{A Functor to Representations of the Auslander Algebra}

Any functor in $\fun(\Lambda)$  is determined by its value on the Auslander generator and its endomorphisms,
so $\fun(\Lambda)$ is equivalent to $\Gamma \bil \fdmod$, this is \cite[Chapitre II, Proposition 2]{gabriel1962}.
Note that the representable functors $(-,M)_{\Lambda}$ correspond to the right $\End(E)$-modules $(E,M)_{\Lambda}$ acted upon by pre-composition, but these may also be viewed as left $\Gamma$-modules.

We will consider the functor
\[
	\alpha \coloneqq \cok (E,-)_{\Lambda} \colon T_2({\Lambda}) \bil \fdmod \rightarrow \Gamma \bil \fdmod,
\]
which was already studied by Auslander and Reiten in \cite{AR1976}.

\begin{rem}
The Gabriel quiver of $\Gamma$ is the opposite quiver of the Auslander-Reiten quiver of $\Lambda \bil \fdmod$ with relations given by the Auslander-Reiten translate.
The indecomposable projective $\Gamma$-modules are represented by the indecomposable objects of $\Lambda \bil \fdmod$.
More precisely, given $M \in \ind (\Lambda)$, then $(E,M)_{\Lambda}$ is the indecomposable projective $\Gamma$-module arising as the projective representation of the opposite of the Auslander-Reiten quiver of $\Lambda \bil \fdmod$ generated at the vertex of $M$.
\end{rem}

\begin{defn}
\label{def:objective}
We call a category a \emph{Krull-Schmidt category} if every object decomposes into a finite direct sum of indecomposable objects in a unique way up to isomorphism.

A functor $F \colon \cA \rightarrow \cB$ between Krull-Schmidt categories is called \emph{objective} if the induced functor $\cA/\ker(F) \rightarrow \cB$ is faithful.
\end{defn}

Our notion of an objective functor is equivalent to that used in \cite{RZ2014}.
For more information on this property we refer to \cite{RZ2015}.

\begin{prp}
\label{prp:alphafunc}
The functor $\alpha$ is full, dense and objective. 
Its kernel is $\add((E \overset{\id}{\rightarrow} E) \oplus (E \rightarrow 0))$.
\end{prp}

\begin{rem}
The indecomposable objects of $\ker(\alpha)$ are either of the form
$(M \overset{\id}{\rightarrow} M)$ or $(M \rightarrow 0)$ for $M \in \ind(\Lambda)$.
Since $\Lambda$ is of finite representation type, say with $m$ indecomposable objects up to isomorphism, 
this means $\ker(\alpha)$ has exactly $2m$ indecomposable objects up to isomorphism.
\end{rem}

\begin{proof}[Proof of Proposition \ref{prp:alphafunc}]
We imitate the proof of \cite[Proposition 3]{RZ2014}.
Let $X$ be an object in $\Gamma \bil \fdmod$, it has a projective presentation
\[
\begin{tikzcd}
(E,M_1)_{\Lambda} \ar[r,"p_1"] & (E,M_0)_{\Lambda} \ar[r,"p_0"] & X \ar[r] & 0.
\end{tikzcd}
\]
By Lemma \ref{lem:representable} there is $f \in (M_1,M_0)_{\Lambda}$ such that $p_1 = (E,f)_{\Lambda}$.
But then $\alpha(f) \simeq X$, so $\alpha$ is dense. 
Let $\Phi \in \Hom_{\Gamma}(X,Y)$ and let $f \in (M_1,M_0)_{\Lambda}$ and $g \in (N_1,N_0)_{\Lambda}$ be such that $\alpha(f) \cong X$ and $\alpha(g) \cong Y$. 
Now $\Phi$ can be extended to a map $(\Phi_1,\Phi_0)$ of the projective presentations of $X$ and $Y$. 
There are $\phi_i$ for $i=0,1$ such that $(E,\phi_i) \cong \Phi_i$.
But then clearly $\alpha(\phi_1,\phi_0) \cong \Phi$, thus $\alpha$ is full.

Clearly $\alpha(M \overset{\id}{\rightarrow} M) \cong 0 \cong \alpha(M \rightarrow 0)$.
Let 
\[
	(g_1,g_0) \in \Hom_{T_2({\Lambda})}((M_1 \overset{f_M}{\rightarrow} M_0),(N_1 \overset{f_N}{\rightarrow} N_0))
\]
be such that $\alpha(g_1,g_0) = 0$. 
We want to show that $(g_1,g_0)$ factors through a $T_2({\Lambda})$-module of the form $(M \overset{\id}{\rightarrow} M) \oplus (N \rightarrow 0)$.

Consider the following commutative diagram:
\[
\begin{tikzcd}[column sep = 3.5em, row sep = 2.5em]
(E,M_1)_{\Lambda} \ar[r,"{(E,f_M)_{\Lambda}}"] \ar[d,swap,"{(E,g_1)_{\Lambda}}"] 
	& (E,M_0)_{\Lambda} \ar[r] \ar[d,"{(E,g_0)_{\Lambda}}"] \ar[dl,dashed,swap,"h'"] & \alpha(f_M) \ar[r] \ar[d,"{{\alpha(g_1,g_0)} = 0}"] & 0\\
(E,N_1)_{\Lambda} \ar[r,swap,"{(E,f_N)_{\Lambda}}"] & (E,N_0)_{\Lambda} \ar[r,swap,"c"] & \alpha(f_N) \ar[r] & 0.
\end{tikzcd}
\]
The rows are projective presentations.
Now $c \circ (E,g_0)_{\Lambda} = 0$ and hence there is $h'$ such that $(E,g_0)_{\Lambda} = (E,f_N)_{\Lambda} \circ h'$.
Since the functor $(E,-)_{\Lambda}$ is full there is a map $h \colon M_0 \rightarrow N_1$ such that $h' = (E,h)_{\Lambda}$ and $g_0 = f_N h$.
Then the following diagram in $\Lambda \bil \fdmod$ is commutative:
\[
\begin{tikzcd}
	M_1 \ar[d,"f_M"] \ar[rr,"{[f_M,\id] }"] & & M_0 \oplus M_1 \ar[rr,"{[h,g_1-hf_M]}"] \ar[d,"{[\id,0]}"] & & N_1 \ar[d,"f_N"]\\
	M_0 \ar[rr,"\id"] & & M_0 \ar[rr,"g_0"] & & N_0.
\end{tikzcd}
\]
Note that the compositions of the rows are $g_1$ and $g_0$, and hence $(g_1,g_0)$ factors through the $T_2({\Lambda})$-module 
$( M_0 \oplus M_1 \overset{[\id,0]}{\longrightarrow} M_0)$.
\end{proof}

\begin{rem}
The functors $\epsilon$ and $\eta$ are faithful and hence objective. 
The composition $\alpha  \eta$ is also objective since it is just a restriction of the objective functor $\alpha$ to an additive subcategory.
Moreover $\alpha \epsilon$ is objective because $\epsilon$ is fully faithful and the image of $\epsilon$ contains all objects of $\ker \alpha$.
\end{rem}

The following corollary of Proposition \ref{prp:alphafunc} describes the composition $\alpha \eta$.

\begin{cor}
\label{cor:subm}
Let $\chi \coloneqq \add(E \overset{\id}{\rightarrow} E)$.
Let $\Gamma \bil \tol$ denote the full subcategory of $\Gamma \bil \fdmod$ consisting of objects of projective dimension $\leq 1$.
The functor $\alpha  \eta$ induces an equivalence of categories
\[
	\cS(\Lambda) / \chi \rightarrow \Gamma \bil \tol.
\]
\end{cor}

\begin{proof}
We know already that $\alpha \eta$ is full and objective and by Proposition \ref{prp:alphafunc} the kernel of $\alpha \eta$ is $\chi$.
It remains to show that the essential image of $\alpha  \eta$ is $\Gamma \bil \tol$.
Let $f \in (M_1,M_0)_{\Lambda}$ be a monomorphism. Since $\Hom$-functors are left-exact, $(E,f)$ is a monomorphism and $\alpha(f)$ has a projective resolution
\[
\begin{tikzcd}[column sep = 2.6em]
0 \ar[r] & (E,M_1)_{\Lambda} \ar[r,"{(E,f)_{\Lambda}}"] & (E,M_0)_{\Lambda} \ar[r] & \alpha(f) \ar[r] & 0.
\end{tikzcd}
\]

Using that $(E,-)_{\Lambda}$ is an equivalence when considered as a functor to $\Gamma \bil \proj$, we see any object in $\Gamma \bil \tol$ has a projective resolution of this form where $f \colon M_1 \rightarrow M_0$ is a monomorphism.
\end{proof}

\begin{rem}
We know $\alpha$ behaves really well with respect to the additive structure on  $T_2(\Lambda) \bil \fdmod$ and $\Gamma \bil \fdmod$, and these are both abelian categories.
However $\alpha$ is far from being exact, in fact it preserves neither epimorphisms nor monomorphisms.
Take for example $\Lambda = k[x]/\langle x^2 \rangle$ and let $_{\Lambda}k$ be the simple $\Lambda$ module.
Consider a monomorphism $f \colon (0 \rightarrow   {_{\Lambda}\Lambda} ) \rightarrow (_{\Lambda}\Lambda \overset{\id}{\rightarrow} {_{\Lambda}\Lambda)}$.
Since $\alpha(_{\Lambda}\Lambda \overset{\id}{\rightarrow} {_{\Lambda}\Lambda}) = 0$ but $\alpha(0 \rightarrow {_{\Lambda}\Lambda})  \neq 0$, $\alpha(f)$ is not a monomorphism. 
Also there is an epimorphism $g \colon (_{\Lambda}\Lambda \overset{\id}{\rightarrow} {_{\Lambda}\Lambda}) \rightarrow (_{\Lambda}\Lambda \rightarrow {_{\Lambda}k})$, but $\alpha(_{\Lambda}\Lambda \rightarrow {_{\Lambda}k}) \neq 0$, thus $\alpha(g)$ is no epimorphism.
\end{rem}

There are several characterizations of the subcategory $\Gamma \bil \tol$, one of which also justifies the notation we use for it.

\begin{prp}
\label{prp:torsionless}
The following are equivalent for an object $X \in \Gamma \bil \fdmod$.
\begin{enumerate}[(i)]
\item
$X$ is in $\Gamma \bil \tol$.
\item
The injective envelope of $X$ is projective.
\item
$X$ is torsionless, i.e. a submodule of a projective module.
\end{enumerate}
\end{prp}

\begin{proof}
(i) $\Longrightarrow$ (ii).
Let $X$ be of projective dimension $\leq 1$, so it has a projective resolution $\begin{tikzcd}[column sep = 1.5em] 0 \ar[r] & P_1 \ar[r,"u"] & P_0 \ar[r] & X \ar[r] & 0. \end{tikzcd}$
Let $v_i \colon P_i \rightarrow I(P_i)$ be the injective envelope of $P_i$ for $i=0,1$ and consider the following diagram:
\[
\begin{tikzcd}
0 \ar[r] & P_1 \ar[r,"u"] \ar[d,"v_1"] & P_0 \ar[r] \ar[d,"v_0"] & X \ar[r] \ar[d,"f"] & 0  \\
0 \ar[r] & I(P_1) \ar[r,"u'"] & I(P_0) \ar[r] & X' \ar[r] & 0.
\end{tikzcd}
\]
Here $X'$ is defined as the module making the diagram commutative with exact rows. 
Since $v_0$ is injective the snake lemma yields a monomorphism $\ker(f) \rightarrow \cok(v_1)$, but since any Auslander algebra has dominant dimension $\geq 2$ we can embed $\cok(v_1)$ into a projective-injective module. 
Thus $\ker(f)$ embeds in a projective-injective module $I(\ker(f))$.
Again using that the dominant dimension of $\Gamma$ is $\geq 2$, we know $I(P_i)$ is projective for $i=0,1$. 
Thus the lower sequence splits and $X'$ is projective-injective. 
The inclusion $\ker(f) \rightarrow I(\ker(f))$ factors through $X$, because $I(\ker(f))$ is injective, and thus we get a monomorphism $X \rightarrow X' \oplus I(\ker(f))$.

(ii) $\Longrightarrow$ (iii).
Clear.

(iii) $\Longrightarrow$ (i).
We have an exact sequence
$
\begin{tikzcd}[column sep = 1.5em]
0 \ar[r] & X \ar[r] & P \ar[r,"\pi"] & C \ar[r] & 0,
\end{tikzcd}
$
where $P$ is projective. Then $C$ has a projective resolution
\[
\begin{tikzcd}
0 \ar[r] & P_2 \ar[r,"p_2"] & P_1 \ar[r,"p_1"] & P \ar[r,"\pi"] & C \ar[r] & 0,
\end{tikzcd}
\]
with $\im(p_1) \cong X$. 
Thus $X$ has a projective resolution of length $\leq 1$.
\end{proof}

We say a module is \emph{divisible} if it is a factor module of an injective module.
We denote the full subcategory of divisible $\Gamma$-modules by $\Gamma \bil \divbl$. We get the following dual statement to Proposition \ref{prp:torsionless}.

\begin{prp}
\label{prp:divisable}
The following are equivalent for an object $X \in \Gamma \bil \fdmod$.
\begin{enumerate}[(i)]
\item
$X$ has injective dimension $\leq 1$.
\item
The the projective cover of $X$ is injective.
\item
$X$ is in $\Gamma \bil \divbl$.
\end{enumerate}
\end{prp}

Later we will have use for the following lemma, which is due to Auslander and Reiten, see \cite[Propositon 4.1]{AR1974}. 
A proof of the version stated here is found in \cite[Section 6]{RZ2014}.

\begin{lemma}
\label{lem:epim}
Let $f$ be a morphism in $\Lambda \bil \fdmod$. Then $f$ is an epimorphism if and only if $((E,P)_{\Lambda},\alpha(f))_{\Gamma} = 0$ for any projective module $P$ in $\Lambda \bil \fdmod$.
\end{lemma}


\subsection{Relative Projective and Injective Objects of $\cS(\Lambda)$}
\label{subsec:frobenius}
The submodule category is additive and by the snake lemma it is closed under extensions. Thus it is an exact subcategory of $T_2(\Lambda) \bil \fdmod$.
The projective and injective objects of $T_2(\Lambda) \bil \fdmod$ are known, a classification can for example be found in \cite[Lemma 1.1]{XZZ2012}.
All projective $T_2(\Lambda)$-modules are a direct sum of modules of the form $(P \overset{\id}{\rightarrow} P)$ or $(0 \rightarrow P)$, where $P$ is a projective $\Lambda$-module.
In particular all projective $T_2(\Lambda)$-modules belong to $\cS(\Lambda)$, and they are the relative projective modules of that exact subcategory.

Dually, the injective $T_2(\Lambda)$-modules are a direct sum of modules of the form $(I \overset{\id}{\rightarrow} I)$ or $(I \rightarrow 0)$, where $I$ is an injective $\Lambda$-module.
The relative injective objects of $\cS(\Lambda)$ can be written as direct sums of objects of the form $(I \overset{\id}{\rightarrow} I)$ or $(0 \rightarrow I)$, with $I$ an injective $\Lambda$-module.

If additionally $\Lambda$ is self-injective, i.e. $\Lambda \bil \fdmod$ is a Frobenius category, the proposition below, found in \cite[Lemma 2.1]{Chen2011}, is an easy consequence.
\begin{prp}
\label{prp:frobenius}
Let $\Lambda$ be a self-injective algebra of finite representation type. 
Then $\cS(\Lambda)$ is a Frobenius category and the projective-injective objects are exactly those in $\add((\Lambda \overset{\id}{\rightarrow} \Lambda) \oplus (0 \rightarrow \Lambda))$.
\end{prp}

\begin{remark}
If $\Lambda$ is self-injective the submodule category $\cS(\Lambda)$ is precisely the full subcategory of Gorenstein projective $T_2(\Lambda)$-modules, cf. \cite[Theorem 1.1]{LZ2010}.
Thus $\eta$ is the inclusion of the Gorenstein projective modules of $T_2(\Lambda) \bil \fdmod$.
\end{remark}


\section{The Auslander Algebra of Self-injective Algebras}
\label{sec:self-inj}

In this section we fix $\Lambda$ as a finite-dimensional basic self-injective $k$-algebra of finite representation type.

Let $\nu \coloneqq D(-, {_{\Lambda}\Lambda})_{\Lambda}$ be the Nakayama functor on $\Lambda \bil \fdmod$. 
Its restriction to projective modules is an equivalence from the projective $\Lambda$-modules to the injective $\Lambda$-modules with inverse $\nu^{-1} \coloneqq (D(-), \Lambda_{\Lambda} )_{\Lambda}$.
Recall that $e$ denotes the idempotent of $\Gamma$ given by the opposite of the projection onto the summand $\Lambda$ of $E$.
Let $\Gamma e$ denote the left ideal generated by $e$. The following lemma describes the projective-injective objects of $\Gamma \bil \fdmod$ explicitly.

\begin{lemma}
\label{lem:proj-inj}
The projective-injective objects of $\Gamma \bil \fdmod$ are precisely the objects of $\add (\Gamma e)$.
Moreover $\Gamma e \cong (E,\Lambda)_{\Lambda} \cong D(\Lambda,E)_{\Lambda}$.
\end{lemma}

\begin{proof}
It is clear that $\Gamma e \cong (E,\Lambda)_{\Lambda}$.
Recall that there is an equivalence $D (_{\Lambda}\Lambda,-)_{\Lambda} \cong (-,\nu  _{\Lambda}\Lambda)_{\Lambda}$ and, since $\Lambda$ is self-injective, $\nu {_{\Lambda}\Lambda} =  {_{\Lambda}\Lambda}$.
Hence $(E,\Lambda)_{\Lambda} \cong D(\Lambda,E)_{\Lambda}$ and by lemmas \ref{lem:representable} and \ref{lem:corepresentable} it is a projective-injective module.

Let $(E,M)_{\Lambda}$ be a projective-injective $\Gamma$-module.
Then every monomorphism $(E,M)_{\Lambda} \rightarrow (E,N)_{\Lambda}$ is a split monomorphism, but that implies any monomorphism ${M \rightarrow N}$ in $\Lambda \bil \fdmod$ is a split monomorphism. Thus $M$ is a projective-injective $\Lambda$-module.
\end{proof}

\begin{rem}
This means the indecomposable projective-injective $\Gamma$-modules are the projective modules at vertices corresponding to indecomposable projective-injective $\Lambda$-modules, when we consider the Gabriel quiver of $\Gamma$ as the opposite of the Auslander-Reiten quiver of $\Lambda \bil \fdmod$.
\end{rem}


\subsection{Recollement}
\label{subsec:recollement}

Notice that $e \Gamma e = \End(\Lambda)^{\op} \cong \Lambda$,  hence $\Lambda$ embeds into $\Gamma$.
Let $\Gamma e \Gamma$ denote the two sided ideal of $\Gamma$ generated by $e$ and denote the quotient $\Gamma / \Gamma e \Gamma$ by $\stGamma$, we call this the \emph{stable Auslander algebra} of $\Lambda$. 

Consider the diagram
\[
\begin{tikzcd}
	\stGamma \bil \fdmod \arrow[rr,"\iota" description] 
	& & \arrow[ll,shift right=1.5ex,"q"'] \arrow[ll,shift left=1.5ex,"p"]  \Gamma \bil \fdmod \arrow[rr,"e" description] 
	& & \arrow[ll,shift right=1.5ex,"l"'] \arrow[ll,shift left=1.5ex,"r"] \Lambda \bil \fdmod
\end{tikzcd}
\]
of functors, where the functors are defined as follows:
\begin{align*}
	&q \coloneqq \Gamma / \Gamma e \Gamma \otimes_{\Gamma} -,  & l \coloneqq \Gamma e \otimes_{\Lambda} -,\\
	&\iota \coloneqq \text{Inclusion},  & e \coloneqq (\Gamma e,-)_{\Gamma},\\
	&p \coloneqq (\Gamma / \Gamma e \Gamma, -)_{\Gamma}, & r \coloneqq (e \Gamma,-)_{\Lambda}.
\end{align*}

This construction goes back to Cline,Parshall and Scott \cite{CPS1988,CPS1988-2}, and it gives a recollement of abelian categories.
By the definition of a recollement we have the following conditions:
\begin{enumerate}[(a)]
\item
The functor $l$ is a left adjoint of $e$ and $r$ is a right adjoint of $e$.
\item
The unit $\id_{\Lambda} \rightarrow el$ and the counit $er \rightarrow \id_{\Lambda}$ are isomorphisms.
\item
The functor $q$ is a left adjoint of $\iota$ and $p$ is a right adjoint of $\iota$.
\item
The unit $\id_{\stGamma} \rightarrow p \iota$ and the counit $q \iota \rightarrow \id_{\stGamma}$ are isomorphisms.
\item
The functor $\iota$ is an embedding onto the full subcategory $\ker(e)$.
\end{enumerate}

\begin{rem}
Since $\Lambda$ is self-injective, $\Gamma$ can be identified with the projective quotient algebra introduced in \cite[Section 5]{CS2017} and the recollement above is the same as the main recollement from \cite[Section 4]{CS2017}.
\end{rem}

We construct the \emph{intermediate extension functor} $c \colon \Lambda \bil \fdmod \rightarrow \Gamma \bil \fdmod$ as follows.
Since the counit $er \rightarrow \id_{\Lambda}$ is an isomorphism we have an inverse $\id_{\Lambda} \rightarrow er$.
If we apply the adjunction $(l,e)$ to the inverse we get a natural transformation $\gamma \colon l \rightarrow r$. 
Then we define $c \coloneqq \im(\gamma)$.

Let us recall the notions of tilting modules and cotilting modules. 
Let $X$ be a $\Gamma$-module, we call $X$ a \emph{tilting} module if the following hold:
\begin{enumerate}[(1)]
\item
The projective dimension of $X$ is at most 1.
\item
$X$ is rigid, i.e. $\ext_{\Gamma}^1(X,X) = 0$.
\item
$X$ has $n$ indecomposable summands where $n$ is the number of indecomposable direct summands of $\Gamma$.  
\end{enumerate}
Dually we say $X$ is \emph{cotilting} if it satisfies (2) and (3) and has injective dimension at most 1.

We say a module $M$ in $A \bil \fdmod$ is \emph{generated} by $N$ if there exists an epimorphism ${N^n \twoheadrightarrow M}$ for some $n \in \N$.
Dually we say $M$ is \emph{cogenerated} by $N$ if there is a monomorphism ${M \hookrightarrow N^n}$ for some $n \in \N$.
We denote by $\gen(N)$ (resp. $\cog(N)$) the full subcategory of modules generated by $N$ (resp. cogenerated by $N$ ).

Consider the $\Gamma$-module $T \coloneqq c(E)$.

\begin{lemma}
\label{lem:cog=tol}
The module $T$ is a tilting and cotilting module. Moreover the following conditions hold.
\begin{enumerate}[$(i)$]
\item
The kernel of $p$ is $\ker(p) = \cog(T) = \Gamma \bil \tol$.
\item
The kernel of $q$ is $ker(q) = \gen(T) = \Gamma \bil \divbl$.
\end{enumerate}
\end{lemma}

\begin{proof}
To see that $T$ is a tilting and cotilting module we refer to \cite[Section 5]{CS2017}. 
There it is also shown that $\ker(p) = \cog(T)$ and $\ker(q) = \gen(T)$.
Since $T$ is tilting, all projective $\Gamma$-modules are in $\cog(T)$. 
Hence $\Gamma \bil \tol$ is contained in $\cog(T)$.
Since $T$ is a tilting module it is of projective dimension at most 1 and hence torsionless by Proposition \ref{prp:torsionless}.
But $\Gamma \bil \tol$ is closed under taking submodules, thus $\cog(T) \subset \Gamma \bil \tol$. This proves $(i)$, the proof of $\gen(T) = \Gamma \bil \divbl$ goes dually.
\end{proof}


\subsection{The Stable Auslander Algebra}

The algebra $\stGamma$ has an alternative description. 
Notice that $\Gamma e \Gamma \subset \Gamma$ is given by all maps in $\End(E)^{\op}$ that factor through a projective-injective $\Lambda$-module. 
Therefore $\stGamma = \underline{\End}_{\Lambda}(E)^{\op}$, the opposite of the endomorphism ring of $E$ in the stable category $\Lambda \bil \stmod$.
Thus $\stGamma \bil \fdmod$ is equivalent to the category of finitely presented additive functors from $(\Lambda \bil \stmod)^{\op}$ to $k$-vector spaces.

The following proposition is classical.
It follows from \cite[Theorem 1.7]{freyd1966} and the fact that every map in a triangulated category is a weak kernel and weak cokernel. 
\begin{prp}[Freyd's Theorem]
\label{prp:tri-frob}
Let $\cT$ be a triangulated category. Then $\fun(\cT)$ is a Frobenius category.
\end{prp}

Since $\Lambda$ is self-injective, $\Lambda \bil \stmod$ is a triangulated category.
Hence the following corollary.

\begin{cor}
\label{cor:Pi-frob}
The category $\stGamma \bil \fdmod$ is a Frobenius category.
\end{cor}


\section{From Submodule Categories to Representations of the Stable Auslander Algebra}
\label{sec:from-to}

Here we follow the story of \cite{RZ2014} in a more general setting for any basic self-injective algebra $\Lambda$ of finite representation type.
We have already studied the functor $\alpha \eta \colon \cS(\Lambda) \rightarrow \Gamma \bil \fdmod $ in Section \ref{sec:subm} and $q \colon \Gamma \bil \fdmod \rightarrow \stGamma \bil \fdmod$ in Section \ref{sec:self-inj}.
We use what we have gathered about those functors to study the compositions
\[
\begin{tikzcd}
	\cS(\Lambda) \ar[r,shift left=1ex,"\eta"] \ar[r,shift right=1ex,"\epsilon"']
	& T_2(\Lambda) \bil \fdmod \ar[r,"\alpha"] 
	& \Gamma \bil \fdmod \ar[r,"q"] 
	& \stGamma \bil \fdmod.
\end{tikzcd}
\]
The functors $F$ and $G$ are given by 
$F \coloneqq q  \alpha \eta$ and $G\coloneqq q \alpha \epsilon$.
The functor $F$ was already studied by Li and Zhang in \cite{LZ2010}.
In \cite{AR1976} Auslander and Reiten considered the composition $G$, based on previous work by Gabriel \cite{gabriel1962}.


\subsection{Induced Equivalences}
\label{subsec:ind-equ}

We have already established that $\eta$ and $\alpha$ as well as the compositions $\alpha \eta$ and $\alpha \epsilon$ are objective.
In corollary \ref{cor:subm} we established that the essential image of $\alpha \eta$ is $\Gamma \bil \tol$. 
For now we shall consider $\alpha \eta$ as a functor to $\Gamma \bil \tol$.
Moreover we write $q_t$ for the restriction of $q$ to $\Gamma \bil \tol$.

\begin{prp}
\label{prp:equ-tol-pi}
The functor $q_t$ is objective, i.e. it induces an equivalence from $\Gamma \bil \tol/ \add(T)$ to $\stGamma \bil \fdmod$.
\end{prp}

\begin{proof}
Since $\ker q = \gen(T)$ and $\Gamma \bil \tol = \cog(T)$ we know that the kernel of $q_t$ is $\cog(T) \cap \gen(T) = \add(T)$.
First we show that the induced functor ${\Gamma \bil \tol / \add(T) \rightarrow \stGamma \bil \fdmod}$ is faithful.
By \cite[Proposition 4.2]{FP2004} there is an exact sequence of functors
\[
\begin{tikzcd}
le \rar["\psi"] & \id_{\Gamma} \ar[r,"\phi"] & \iota q \rar & 0.
\end{tikzcd}
\]
Let $Y \in \Gamma \bil \tol$, there is an epimorphism $\phi_Y \colon Y \rightarrow \iota q(Y)$ and the morphism ${\psi_{Y} \colon l e (Y) \rightarrow Y}$ factors through $\ker(\phi_Y)$ via an epimorphism, in particular $\ker(\phi_Y)$ is in $\gen(T)$.
Since $\ker(\phi_Y)$ is a submodule of $Y$ it belongs to $\cog(T)$, thus $\ker(\phi_Y) \in \add(T)$. 
Now let $f \colon X \rightarrow Y$ be a morphism in $\Gamma \bil \tol$ such that $q_t(f) = 0$. 
Then $\phi_Y  f = 0$ and thus $f$ factors through $\ker(\phi_Y)$.

To show $q_t$ is full we consider the adjoint pair $(q,\iota)$. Let $X,Y \in \Gamma \bil \tol$, we want to show the map $q_{XY}$ induced by the functor $q$ in the following sequence is surjective.
\[
\begin{tikzcd}
	(X,Y)_{\Gamma} \ar[r,"q_{XY}"] & (q X, q Y)_{\stGamma} \ar[r,"\Phi"] & (X,\iota q Y)_{\Gamma}.
\end{tikzcd}
\]
Here $\Phi$ is the isomorphism given by the adjunction.
Let $\phi_Y \coloneqq \Phi(\id_{q Y})$, by \cite[Proposition 4.2]{FP2004} this is an epimorphism, so we get an exact sequence
\[
\begin{tikzcd}
	0 \rar & K \rar & Y \ar[r,"\phi_Y"] & \iota q Y \rar & 0.
\end{tikzcd}
\]
By our argument above we know $K \in \add(T)$. 
Apply $(X,-)_{\Gamma}$ to the exact sequence above and get the exact sequence
\[
\begin{tikzcd}[column sep = 2.5em]
	0 \rar & (X,K)_{\Gamma} \rar & (X,Y)_{\Gamma} \ar[r,"{(X,\phi_Y)}"] & (X, \iota q Y)_{\Gamma} \rar & \ext_{\Gamma}^1(X,K). 
\end{tikzcd}
\]
Since $X \in \cog(T)$  and $K \in \add (T)$ we have $\ext^1(X,K) = 0$, thus $(X,\phi_Y) = \Phi q_{XY}$ is an epimorphism. 
Since $\Phi$ is an isomorphism that implies $q_{XY}$ is an epimorphism.

To show denseness we adapt the proof of \cite[Proposition 5]{RZ2014}.
Let $\underline{X} \in \stGamma \bil \fdmod$ and write $X \coloneqq \iota(\underline{X})$. 
We let $u \colon X \rightarrow I(X)$ be the injective envelope and $p \colon PI(X) \rightarrow I(X)$ be a projective cover with kernel $K$.
We get an induced diagram
\[
\begin{tikzcd}
0 \rar & K \ar[r,"v'"] \ar[xshift=0.3ex,d,-] \ar[xshift=-0.3ex,d,-]& Y \ar[r,"p'"] \ar[d,"u'"] & X \rar \ar[d,"u"] & 0\\
0 \rar & K \ar[r,"v"] & PI(X) \ar[r,"p"] & I(X) \rar & 0.
\end{tikzcd}
\]
Since $u'$ is a monomorphism, $Y$ embeds into a projective-injective $\Gamma$-module, and hence $Y \in \Gamma \bil \tol$.
Moreover $K$ has injective dimension at most 1, hence $q_t(K) = 0$.
By the defining properties of a recollement we have $q_t(X) \cong q \iota(\underline{X}) \cong \underline{X}$, and $q_t$ is right-exact because $q$ is a left adjoint. 
It follows that $q_t(Y) \cong q_t(X) \cong \underline{X}$.
We have shown that $q_t$ is dense.
\end{proof}

We know $\alpha \eta$ is objective and dense when considered as a functor to $\Gamma \bil \tol$ and thus the functor $F$ is objective.
Moreover $F$ is full and dense because $q_t$ and $\alpha \eta$ are full and dense.

Let $f \in T_2(\Lambda) \bil \fdmod$, then Lemmas \ref{lem:epim} and \ref{lem:proj-inj} imply that $f$ is an epimorphism if and only if
$e(\alpha(f)) = (\Gamma e,\alpha(f))_{\Gamma} = 0$.
This means the essential image of $\alpha \epsilon$ is $\ker(e)$,
but we can identify $\ker(e)$ with $\stGamma \bil \fdmod$ via $\iota$.
We have established $\alpha \epsilon$ is an objective functor so we conclude $G$ is objective.
The functor $G$ is also full and dense because $\alpha \epsilon$ is full and dense when considered as a functor to $\ker (e)$.

Now we can prove our first main theorem.

\induced*

\begin{proof}
The indecomposable objects of $\cV$ are $(M \overset{\id}{\rightarrow} M)$ and $(0 \rightarrow M)$ for any $M \in \ind(\Lambda)$. Hence $\cV$ clearly has $2m$ indecomposable objects up to isomorphism.
The indecomposable objects of $\cU$ are $(M \overset{\id}{\rightarrow} M)$ and the injective envelope $(M \rightarrow I(M))$ for each  $M \in \ind(\Lambda)$, as well as the objects $(0 \rightarrow I)$ for each injective object $I \in \ind(\Lambda)$.
Since the objects $(I \overset{\id}{\rightarrow} I)$ for $I \in \ind(\Lambda)$ injective appear twice in this list, $\cU$ has $2m$ indecomposable objects up to isomorphism.

Next we prove $(i)$. 
Let $(M \overset{f}{\rightarrow} N) \in \cS(\Lambda)$, and assume $F(f) = 0$. 
Consider the diagram
\[
\begin{tikzcd}[row sep = 1.5em]
0 \ar[r] & P \ar[r,"\sim"] \ar[d] & P \ar[r] \ar[d] & 0 \ar[d] &\\
0 \ar[r] & (E,M)_{\Lambda} \ar[r,"{(E,f)_{\Lambda}}"] \ar[d] & (E,N)_{\Lambda} \ar[r] \ar[d] 
& \alpha(f) \ar[xshift=0.3ex,d,-] \ar[xshift=-0.3ex,d,-] \ar[r] & 0 \\
0 \ar[r] & P_1 \ar[r,"p_1"] & P_0 \ar[r] & \alpha(f) \ar[r] & 0.
\end{tikzcd}
\]
Here the bottom row is a minimal projective resolution and all rows and columns are exact, thus $P$ is projective and the first two columns are split exact sequences of projective modules.
Since $\Lambda \bil \fdmod$ is equivalent to the full subcategory of projective $\Gamma$-modules this shows $f$ is a direct sum of an isomorphism $(M' \overset{f'}{\rightarrow} N')$, corresponding to $P \simeq P$, and a monomorphism $(M'' \overset{f''}{\rightarrow} N'')$, corresponding to the map $p_1$. 
Now $F(f) = q(\alpha(f)) = 0$ if and only if $P_0$ is projective-injective by Lemma \ref{lem:cog=tol}, which is if and only if $N''$ is projective-injective by Lemma \ref{lem:proj-inj}.

We already know $F$ is objective, and thus the functor $\cS(\Lambda)/ \cU \rightarrow \stGamma \bil \fdmod$ induced by $F$ is faithful.  
Since $F$ is full and dense the induced functor is also full and dense. 

Now to $(ii)$.
The kernel of $\alpha$ is  $\add( ( E \overset{\id}{\rightarrow} E) \oplus (E \rightarrow 0))$.
But  
\[\epsilon(( E \overset{\id}{\rightarrow} E) \oplus (0 \rightarrow E)) = ( E \overset{\id}{\rightarrow} E) \oplus (E \rightarrow 0),
\] 
hence $\cV = \add(( E \overset{\id}{\rightarrow} E) \oplus (0 \rightarrow E)) = \ker (\alpha \epsilon)$. Moreover the restriction of $q$ to the essential image $\ker(e)$ of $\alpha \epsilon$ is an equivalence by the defining properties of a recollement. 
We have shown $G$ is full, dense and objective, thus $(ii)$ holds.
\end{proof}


\subsection{Interplay with Triangulated Structure}
\label{subsec:triangulated}

We have already established that the categories $\cS(\Lambda)$ and $\stGamma \bil \fdmod$ are Frobenius categories.
Then it is natural to ask whether the triangulated structure of the stable category $\stGamma \bil \stmod$ interacts nicely with that functors $F$ and $G$.

Let $\pi \colon \stGamma \bil \fdmod \rightarrow \stGamma \bil \stmod$ be the projection to the stable category.
We denote the syzygy functor on $\stGamma \bil \stmod$ by $\Omega$. 
The following was proven in a special case in \cite[Section 7]{RZ2014}, and we prove this more general statement analogously.

\difference*

\begin{proof}
Let $(L \overset{f}{\rightarrow} M)$ be an object in $\cS(\Lambda)$.
We have the corresponding exact sequence
\[
\begin{tikzcd}
0 \rar & L \ar[r,"f"] & M \ar[r,"g"] & N \rar & 0.
\end{tikzcd}
\]
Notice that $g = \epsilon(f)$.
Apply $(E,-)_{\Lambda}$ to this sequence and obtain the exact sequence
\[
\begin{tikzcd}
0 \rar & (E,L)_{\Lambda} \ar[r,"{(E,f)}"] & (E,M)_{\Lambda} \ar[r,"{(E,g)}"] & (E,N)_{\Lambda} .
\end{tikzcd}
\]
The cokernel of $(E,f)$, and hence the image of $(E,g)$, is by definition $\alpha \eta (f)$.
Also the cokernel of $(E,g)$ is $\alpha \epsilon(f)$.
Thus we get an exact sequence
\[
\begin{tikzcd}[column sep = 2.8em]
	0 \rar & \alpha \eta(f) \ar[r,"{\im(E,g)}"] & (E,N)_{\Lambda} \rar & \alpha \epsilon(f) \rar & 0.
\end{tikzcd}
\]

From \cite[Proposition 4.2]{FP2004} we know there is an exact sequence of functors 
$
\begin{tikzcd}[column sep = 1.5em]
	l  e \rar & \id_{\Gamma} \rar &  \iota q \rar & 0.
\end{tikzcd}
$
We obtain a commutative diagram with exact rows and columns:
\[
\begin{tikzcd}[column sep = 1.9em]
 & l  e  \alpha \eta(f) \ar[rr,"\phi"] \dar & & l  e(E,N)_{\Lambda} \dar & & 0 \dar & & \\
0 \rar & \alpha  \eta(f) \ar[rr,"{\im(E,g)}"] \dar & & (E,N)_{\Lambda} \ar[rr] \dar & & \alpha \epsilon(f) \dar \rar & 0 \\
 & \iota F(f)  \ar[rr,"{\iota q (\im(E,g))}"] \dar & & \iota q(E,N)_{\Lambda} \ar[rr] \dar & & \iota G(f) \rar \dar & 0 \\
 & 0 & & 0 & & 0 & 
\end{tikzcd}
\]
Since $e$ is exact and $e \alpha \epsilon = 0$ the map $\phi$ is an isomorphism.
But then we can extend the top row to a short exact sequence and apply the snake lemma to see that $\iota q(\im(E,g))$ is a monomorphism.

Since $\iota$ is fully faithful and exact this implies we have the following exact sequence in $\stGamma \bil \fdmod$:
\[
\begin{tikzcd}
	0 \rar & F(f) \rar & q((E,N)_{\Lambda}) \rar & G(f) \rar & 0.
\end{tikzcd}
\]
Now $\iota$ preserves epimorphisms and $q$ is its left adjoint, thus $q$ preserves projective objects.
We know $(E,N)_{\Lambda}$ is a projective $\Gamma$-module and hence $q((E,N)_{\Lambda})$ is projective, this shows $\pi F(f) \cong \Omega \pi G(f)$ in $\stGamma \bil \stmod$.
\end{proof}

\begin{rem}
By Proposition \ref{prp:frobenius} $\cS(\Lambda)$ is also a Frobenius category, so the stable category $\stsubm{}{\Lambda}$
is a triangulated category.
Hence one might ask whether $F$ and $G$ induce a triangle functor from $\stsubm{}{\Lambda}$ to $\stGamma \bil \stmod$.
However all maps factoring through projective objects in $\cS(\Lambda)$ factor through both $\cU$ and $\cV$, thus any induced triangle functor would have to  factor through the abelian category $\stGamma \bil \fdmod$, which renders any such functor trivial.
\end{rem}


\section{Auslander Algebras of Nakayama Algebras}
\label{sec:nakayama}

A finite length module is said to be \emph{uniserial} if it has a unique composition series.
We say an algebra $A$ is uniserial, or a \emph{Nakayama algebra}, if all indecomposable $A$-modules have a unique composition series.
In this section we prove Theorem \ref{thm:uniserial}.

\uniserial*

The only if part is proven in Subsection \ref{subsec:only-if}, but before that we describe a quasi-hereditary structure with the properties from Theorem \ref{thm:uniserial} for the Auslander algebras of Nakayama algebras.
First, however, we consider the example of self-injective Nakayama algebras over an algebraically closed field explicitly, to get some picture of the situation.


\subsection{Self-injective Nakayama Algebras}
\label{subsec:nakayama}
The classification of Nakayama algebras over algebraically closed fields is well known, and can for example be found in \cite[V.3]{ASS2006}.
We recall the self-injective case in this subsection to get an explicit description of an example.
We are particularly interested in the self-injective Nakayama algebras, so we consider the structure of their Auslander algebras in explicit terms here.
Let $\tilde{\A}_m$ denote the quiver with vertices $\Z / m \Z$ and arrows $i \rightarrow i+1$ for all $i \in \Z/ m\Z$.
Write $k\tilde{\A}_m$ for the path algebra of this quiver and let $J(k \tilde{\A}_m )$ denote the ideal generated by the arrows.
For $m,N \in \N$ we define $A(m,N) \coloneqq k \tilde{\A}_m / J(k \tilde{\A}_m)^{N+1}$, these are precisely the basic connected self-injective Nakayama algebras.
Notice that $A(1,N) \cong k[x]/\langle x^{N+1} \rangle$, and hence the case studied in detail in \cite{RZ2014} is included.
We parametrize the simple $A(m,N)$ modules by the vertices $j \in \Z / m \Z$ of $\tilde{\A}_m$.
The category $A(m,N) \bil \fdmod$ has indecomposable objects $[i]_j$ for $j \in \Z / m \Z$ and $i = 1,...,N+1$, where $\soc([i]_j) = S(j)$ and $[i]_j$ has Loewy length $i$.

To get an idea of the general shape of the Auslander-Reiten quiver, take for example the Auslander-Reiten quiver of $A(4,3)$ in figure \ref{fig:A(4,3)}.
\begin{figure}
\[
\xymatrix@=1.6em{
 \cdots & \ar@{.}[l] [1]_1 \ar[dr] & & \ar@{.>}[ll] [1]_0 \ar[dr] & & \ar@{.>}[ll] [1]_3 \ar[dr] & & \ar@{.>}[ll] [1]_2 \ar[dr] & \ar@{.>}[l] \cdots \\
  {[2]_2} \ar[dr] \ar[ur] & & \ar@{.>}[ll] [2]_1 \ar[dr] \ar[ur] & & \ar@{.>}[ll] [2]_0 \ar[ur] \ar[dr] & & \ar@{.>}[ll] [2]_3 \ar[ur] \ar[dr] & & \ar@{.>}[ll] [2]_2 \\
 \cdots & \ar@{.}[l] [3]_2 \ar[ur] \ar[dr] & & \ar@{.>}[ll] [3]_1 \ar[ur] \ar[dr] & & \ar@{.>}[ll] [3]_0 \ar[ur] \ar[dr] & & \ar@{.>}[ll] [3]_3 \ar[dr] \ar[ur] & \ar@{.>}[l] \cdots\\
{[4]_3} \ar[ur] & & {[4]_2} \ar[ur] & & {[4]_1} \ar[ur] & & {[4]_0} \ar[ur] & & {[4]_3}
}
\]
\caption{Auslander-Reiten quiver of $A(4,3)$}
\label{fig:A(4,3)}
\end{figure}
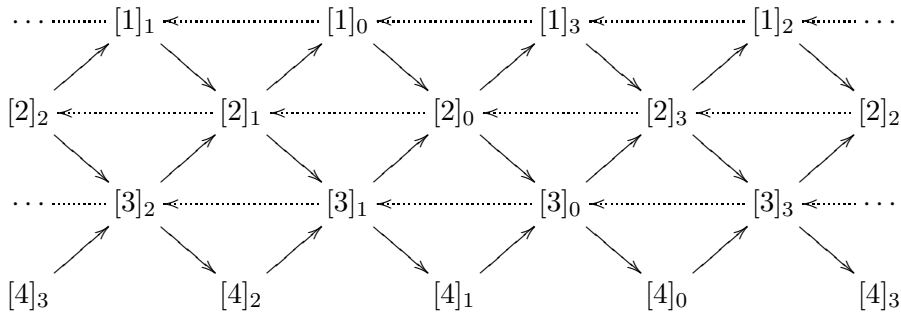
The Gabriel quiver of $\Gamma$ is the opposite of this quiver.

We may consider $A(m,N) \bil \fdmod$ as a $\Z/m\Z$-fold cover of $A(1,N) \bil \fdmod$. 
Namely, if we give $A(1,N) = k[x]/\langle x^{N+1} \rangle$ the $\Z$-grading given by monomial degrees, it induces a $\Z/m\Z$ grading and the categories $k[x] / \langle x^{N+1} \rangle \bil \fdmod^{\Z/m\Z}$ and $A(m,N) \bil \fdmod$ are isomorphic.  


\subsection{Quasi-hereditary Algebras}
\label{subsec:QH}

We follow the approach of \cite{DR1992} to quasi-hereditary algebras.
Let $A$ be a basic finite-dimensional algebra and let $\Xi$ be a set parametrizing the isomorphism classes of simple objects in $A \bil \fdmod$. 
We write $S(\xi)$ for the simple $A$-module corresponding to $\xi \in \Xi$, $P(\xi)$ for the projective cover of $S(\xi)$, and $I(\xi)$ for the injective envelope of $S(\xi)$.
Give $\Xi$ a partial ordering $(\Xi,\leq)$. For each $\xi \in \Xi$ the \emph{standard} module $\Delta(\xi)$ is the maximal factor module of $P(\xi)$ such that for all composition factors $S(\rho)$ of $\Delta(\xi)$, we have $\rho \leq \xi$.
Dually the costandard module $\nabla(\xi)$ is the maximal submodule of $I(\xi)$ such that for all composition factors $S(\rho)$, we have $\rho \leq \xi$.
A partial ordering on the simple $A$-modules is called \emph{adapted} if for every $A$-module $X$ with $\Top(X) = S(\xi)$ and $\soc(X) = S(\xi')$, such that $\xi$ and $\xi'$ are incomparable, there is $\rho \in \Xi$ such that $\rho > \xi$ and $S(\rho)$ is in the composition series of $X$. If the partial ordering is adapted, then all the standard and costandard modules are the same for any refinement of the partial ordering cf. \cite[Section 1]{DR1992}.

The full subcategory of standard (resp. costandard) modules in $A\bil \fdmod$ will be denoted by $\Delta$ (resp. $\nabla$), and $\cF(\Delta)$ (resp. $\cF(\nabla)$) denotes the full subcategory of all objects that have a filtration by standard (resp. costandard) modules. 
If the standard modules are Schurian and $A \in \cF(\Delta)$, we say $A$ is \emph{quasi-hereditary}.
The \emph{characteristic tilting module} of a quasi-hereditary algebra is determined up to isomorphism as the basic module $C$ such that $\cF(\Delta) \cap \cF(\nabla) = \add (C)$. By \cite[Proposition 3.1]{DR1992} it is indeed a tilting module.

Recall that a \emph{left strongly} quasi-hereditary algebra is a quasi-hereditary algebra such that all standard modules have projective dimension at most 1, or equivalently that $C$ has projective dimension at most 1.
Ringel has shown that any Auslander algebra of a representation-finite algebra has a left strongly quasi-hereditary structure cf. \cite[Section 5]{ringel2010}. 

The isomorphism classes of simple $\Gamma$-modules are in a canonical bijection with the isomorphism classes of $\ind(\Lambda)$, denoted by $\ind(\Lambda)/ \!\! \sim$.
For any $M \in \ind(\Lambda)$, let $p_M \colon E \rightarrow M$ be the projection, and $i_M \colon M \rightarrow E$ be the inclusion. 
Then $M$ corresponds to the idempotent $(i_M p_M)^{\op} \in \Gamma$, which corresponds to an isomorphism class of simple $\Gamma$-modules.
We use this bijection to parametrize the simple $\Gamma$-modules.
For $M \in \ind(\Lambda)$ we let $[M]$ denote its isomorphism class in $\ind(\Lambda)/ \!\! \sim$, although we write $S(M),P(M),I(M),\Delta(M),\nabla(M)$ resp. instead of $S([M]),P([M]),I([M]),$ $\Delta([M]),\nabla([M])$ resp.


\subsection{Auslander Algebras of Nakayama Algebras}
\label{subsec:QH-Naka}
Let $\Lambda$ be a Nakayama algebra and let $\Gamma$ be its Auslander algebra.
Let $\ell(M)$ denote the Loewy length of a $\Lambda$-module $M$.
We consider a partial ordering on $\ind(\Lambda)/ \!\! \sim$ given by the Loewy length:
For $M,N \in \ind(\Lambda)$, say $[M] > [N]$ if $\ell(M) < \ell(N)$, but $[M],[N]$ are incomparable if $\ell(M) = \ell(N)$.
\begin{rem}
Notice that modules with greater Loewy-length are smaller in our partial ordering.
Thus the simple modules are maximal.
\end{rem}

Let $X$ be an indecomposable $\Gamma$-module with $\Top(X) = S(M)$ and $\soc(X)=S(N)$.
If $M \ncong N$ and $\ell(M) = \ell(N)$ there is a non-trivial map $f \colon N \rightarrow M$ such that for every indecomposable summand $M'$ of $\im(f)$, $S(M')$ is in the composition series of $X$. 
We have $\ell(M)> \ell(M')$ for every such summand $M'$ of $\im(f)$, i.e. $[M'] > [M]$. Thus our partial order is adapted.

Let $M \in \ind(\Lambda)$. 
If $M$ is simple then there are no non-trivial homomorphisms from other simple $\Lambda$ modules to $M$.
Hence $\Hom_{\Gamma}(P(N),P(M)) = 0$ for all simple $N \not \cong M$, which implies $\Delta(M) = P(M)$.
Similarly we have $\Hom_{\Gamma}(I(M),I(N)) = 0$ for all simple $N \not \cong M$, and thus $\nabla(M) \cong I(M)$.

Now $\Lambda$ is uniserial, so if $M$ is not simple it has a unique maximal proper submodule $M'$, and clearly $\ell(M') = \ell(M) - 1$.
Recall that $P(M) \cong (E,M)_{\Lambda}$.
Let $N \in \ind(\Lambda)$, for any map in $(N,M')_{\Lambda}$, composition with the inclusion of $M'$ in $M$ gives a map in $(N,M)_{\Lambda}$. 
In this way $P(M')$ embeds in $P(M)$ as a $\Gamma$-submodule.
Any non-surjective map to $M$ factors through $M'$, in particular, if $N \in \ind(\Lambda)$ with $\ell(N) < \ell(M)$, then any map in $(N,M)_{\Lambda}$ factors through $M'$.
This shows that $\Delta(M) \cong P(M)/P(M')$ and thus $\Delta(M)$ has projective dimension 1. 
Consequently all modules in $\cF(\Delta)$ have projective dimension at most 1.

We proceed in a similar way for the costandard modules.
If $M$ is non-simple, then it has a unique maximal proper factor module $M''$. 
The projection $M \rightarrow M''$ induces an epimorphism $I(M) \rightarrow I(M'')$. 
We know any non-injective map from $M$ to $N$ factors through $M''$.
In particular, if $\ell(N) < \ell(M)$ and $N \in \ind(\Lambda)$, then any map in $(M,N)_{\Lambda}$ factors through the projection $M \rightarrow M''$. 
Thus the kernel of the map $I(M) \rightarrow I(M'')$ has no composition factors $S(N)$ such that $\ell(N) < \ell(M)$. 
Together this implies that we have an exact sequence
\[
\begin{tikzcd}
0 \ar[r] & \nabla(M) \ar[r] & I(M) \ar[r] & I(M'') \ar[r] & 0.
\end{tikzcd}
\]
In particular $\nabla(M)$ has injective dimension 1 and thus any module in $\cF(\nabla)$ has injective dimension at most 1.
Taking everything together we get the following proposition.

\begin{prp}
\label{prp:coincide}
Let $\Gamma$ be the Auslander algebra of a Nakayama algebra. The partial ordering above gives $\Gamma$ a quasi-hereditary structure with $\cF(\Delta) = \Gamma \bil \tol$ and $\cF(\nabla) = \Gamma \bil \divbl$.
\end{prp}
\begin{proof}
The $\Gamma$-module $_\Gamma\Gamma$ is in $\cF(\Delta)$ and all the standard modules are Schurian.
Thus our partial ordering gives a quasi-hereditary structure on $\Gamma$.
Since all objects of $\cF(\Delta)$ have projective dimension at most 1 we see $\cF(\Delta) \subset \Gamma \bil \tol$.
Also all the costandard modules have injective dimension at most 1, so by \cite[Lemma 4.1*]{DR1992} and Proposition \ref{prp:torsionless} we have $\Gamma \bil \tol \subset \cF(\Delta)$.

We show $\cF(\nabla) = \Gamma \bil \divbl$ dually using \cite[Lemma 4.1]{DR1992} and Proposition \ref{prp:divisable}.
\end{proof}


\subsection{Other Representation-finite Algebras}
\label{subsec:only-if}

Quasi-hereditary structures on Auslander algebras of representation-finite algebras that have the property given in Proposition \ref{prp:coincide} are rare in general.
Indeed, the examples illustrated in Subsection \ref{subsec:QH-Naka} are the only cases.

\begin{prp}
\label{prp:uniserial}
Let $\Lambda$ be a basic representation-finite algebra and let $\Gamma$ be its Auslander algebra.
If $\Gamma$ has a quasi-hereditary structure such that the $\Delta$-filtered modules coincide with $\Gamma \bil \tol$,
then $\Lambda$ is uniserial.
\end{prp}

\begin{proof}
We keep the notation from Subsection \ref{subsec:QH}. 
Let $\Gamma$ have a quasi-hereditary structure such that $\cF(\Delta) = \Gamma \bil \tol$.
It suffices to show that all indecomposable projective and all indecomposable injective $\Lambda$-modules have a unique composition series.
Let $M \in \ind(\Lambda)$ be a submodule of an indecomposable injective module. 
Let $\Delta(M)$ be the standard module generated by $(E,M)_{\Lambda}$.
By assumption we have a projective resolution
\[
\begin{tikzcd}
0 \rar & P_1 \ar[r,"\pi_1"] & (E,M)_{\Lambda} \rar["\pi_0"] & \Delta(M) \rar & 0.
\end{tikzcd}
\]
Since $\Gamma \bil \proj$ is equivalent to $\Lambda \bil \fdmod$ we get a monomorphism $f \colon N \rightarrow M$ in $\Lambda \bil \fdmod$ such that $\pi_1 = (E,f)_{\Lambda}$.
Let $M'$ be any proper submodule of $M$ and let $\iota$ denote its inclusion. 
Then $\alpha(\iota)$ is $\Delta$-filtered, hence it has $\Delta(M)$ as a factor module.
Thus there is a commutative diagram
\[
\begin{tikzcd}
0 \rar & (E,M')_{\Lambda} \ar[r,"{(E,\iota)}"] \ar[d,dashed,"\psi"] & (E,M)_{\Lambda} \ar[r] \ar[xshift=0.3ex,d,-] \ar[xshift=-0.3ex,d,-] & \alpha(\iota) \ar[d,twoheadrightarrow] \rar & 0\\
0 \rar & (E,N)_{\Lambda} \ar[r,"\pi_1"] & (E,M)_{\Lambda} \ar[r,"\pi_0"] & \Delta(M) \rar & 0.
\end{tikzcd}
\]
We get an induced map $\psi \colon (E,M')_{\Lambda} \rightarrow (E,N)_{\Lambda}$ making the diagram above commutative.
This yields a monomorphism $g \colon M' \rightarrow N$ such that $\iota = f g$. 
If we identify $N$ with its image in $M$ via $f$ this shows every proper submodule $M'$ of $M$ is a submodule of $N$.
Since $N$ is a submodule of an indecomposable injective module it is also indecomposable.
This shows that any indecomposable injective $\Lambda$-module has a unique composition series.
Dually, using that $\Gamma \bil \divbl = \cF(\nabla)$ we can show all indecomposable projective $\Lambda$-modules have a unique composition series.
\end{proof}

Combining Proposition \ref{prp:uniserial} with Proposition \ref{prp:coincide} now yields Theorem \ref{thm:uniserial}.

Given a quasi-hereditary structure on $\Gamma$ the condition $\cF(\Delta) = \Gamma \bil \tol$ is the same as condition $(i)$ in \cite[Lemma 4.1]{DR1992} combined with condition $(iv)$ in \cite[Lemma 4.1*]{DR1992}.
Thus all the conditions of these two lemmas hold, in particular $\cF(\nabla) = \Gamma \bil \divbl$ and thus $\add(T) = \add(C)$. 
Since both $T$ and $C$ are basic this implies they are isomorphic. 
Conversely, if $T \cong C$, then $\cF(\Delta) = \cog(C) = \cog(T) = \Gamma \bil \tol$.
Hence Theorem \ref{thm:uniserial} yields the following corollary.

\begin{cor}
Let $\Lambda$ be a basic self-injective algebra of finite representation type and
let $\Gamma$ be the Auslander algebra of $\Lambda$. 
We let $T = c(E)$ be the canonical tilting and cotilting module as defined in Subsection \ref{subsec:recollement}.
Then $T$ is a characteristic tilting module of a quasi-hereditary structure on $\Gamma$ if and only if $\Lambda$ is a Nakayama algebra.
\end{cor}


\section*{Acknowledgements}
I want to thank Julia Sauter,  my de facto advisor, who suggested this topic 
to me, and provided advice and ideas throughout the writing of this article.
I would also like to thank Claus M. Ringel for enlightening discussions and Henning Krause for useful comments.
Finally I thank the referee for reading the article carefully and suggesting several improvements.



\end{document}